\documentclass[12pt]{article}
\usepackage[a4paper]{anysize}\marginsize{2.0cm}{2.0cm}{2.0cm}{2.0cm}
\pdfpagewidth=\paperwidth \pdfpageheight=\paperheight
\usepackage{amsfonts,amssymb,amsthm,amsmath,eucal,tabu,url}
\usepackage{pgf}
 \usepackage{array}
 \usepackage{pstricks}
 \usepackage{pstricks-add}
 \usepackage{pgf,tikz}
 \usetikzlibrary{automata}
 \usetikzlibrary{arrows}
 \usepackage{indentfirst}
\theoremstyle{plain}
\newtheorem{thm}{Theorem}[section]
\newtheorem{theorem}[thm]{Theorem}

\newtheorem{lemma}[thm]{Lemma}

\theoremstyle{definition}
\newtheorem{definition}[thm]{Definition}
\newtheorem{remark}[thm]{Remark}
\newtheorem{example}[thm]{Example}

\newtheorem{thevarthm}[thm]{\varthmname}

\newenvironment{varthm*}[1]{\trivlist\item[]{\bf #1.}\it}{\endtrivlist}

\renewcommand\geq{\geqslant}

\renewcommand\leq{\leqslant}

\newcommand\be{\begin{eqnarray*}}
\newcommand\ee{\end{eqnarray*}}

\newcommand\newop[2]{\def#1{\mathop{\rm #2}\nolimits}}
\newop\log{log}
\newop\ord{ord}
\newop\Gal{Gal}
\newop\SL{SL}
\newop\Bl{Bl}
\newop\mult{mult}
\newop\mass{mass}
\newop\div{div}
\newop\codim{codim}
\newop\sing{sing}
\newop\vdim{vdim}
\newop\edim{edim}
\newop\Ass{Ass}
\newop\size{size}
\newop\reg{reg}
\newop\satdeg{satdeg}
\newop\supp{supp}
\newop\Neg{Neg}
\newop\Nef{Nef}
\newop\Nefh{Nef_H}
\newop\Eff{Eff}
\newop\Zar{Zar}
\newop\MB{MB}
\newop\MBxC{MB\mathit{(x,C)}}
\newop\NnB{NnB}
\newop\Bigg{Big}
\newop\Effbar{\overline{\Eff}}

\def\keywordname{{\bfseries Keywords}}%
\def\keywords#1{\par\addvspace\medskipamount{\rightskip=0pt plus1cm
\def\and{\ifhmode\unskip\nobreak\fi\ $\cdot$
}\noindent\keywordname\enspace\ignorespaces#1\par}}
\def\subclassname{{\bfseries Mathematics Subject Classification
(2000)}\enspace}
\def\subclass#1{\par\addvspace\medskipamount{\rightskip=0pt plus1cm
\def\and{\ifhmode\unskip\nobreak\fi\ $\cdot$
}\noindent\subclassname\ignorespaces#1\par}}

\begin{document}
\title{Hirzebruch-Kummer covers of algebraic surfaces}
\author{Piotr Pokora}
\date{\today}
\maketitle
\thispagestyle{empty}
\begin{abstract}
The aim of this paper is to show that using some natural curve arrangements in algebraic surfaces and Hirzebruch-Kummer covers one cannot construct new examples of ball-quotients, i.e., minimal smooth complex projective surfaces of general type satisfying equality in the Bogomolov-Miyaoka-Yau inequality.
\subclass{14C20, 14J29, 14N20}
\end{abstract}
\section{Introduction}
In his pioneering papers \cite{Hirzebruch,Hirzebruch1}, Hirzebruch constructed some new examples of algebraic surfaces which are ball-quotients, i.e., algebraic surfaces for which the universal cover is the $2$-dimensional unit ball. Alternatively, these are minimal smooth complex projective surfaces of general type satisfying equality in the Bogomolov-Miyaoka-Yau inequality~\cite{M84}:

$$K_{X}^{2} \leq 3e(X),$$
where $K_{X}$ denotes the canonical divisor and $e(X)$ is the topological Euler characteristic. The key idea of Hirzebruch, which enabled constructing new examples of ball-quotients, is that one can consider abelian covers of the complex projective plane branched along line arrangements \cite{BHH87}. The same idea was used by the author in \cite{Pokora}, where the main result tells us that for $d$-configurations of curves (configurations of smooth irreducible curves in the complex projective plane, each irreducible component has the same degree $d \geq 2$, and all intersection points are transversal) the associated Hirzebruch-Kummer construction in most cases does not provide new examples of ball-quotients. On the other side, it is worth pointing out that  Hirzebruch's idea can be also investigated from a different point of view focusing on $H$-indices, and we refer to \cite[Section 4]{Roulleau} for details. Looking at Hirzebruch's construction, one might hope that suitably adapted it can yield new examples of ball-quotients coming from certain curve arrangements in algebraic surfaces. That was a hope of F. Hirzebruch and his research group in the 1980s when several PhD theses were written. For instance, a nice project was conducted by Bruce Hunt \cite{Hunt}, where he considered a 3-dimensional analog of Hirzebruch's construction. However, it turned out that using hyperplane arrangements in the three dimensional complex projective space one cannot construct new examples of ball-quotients using the so-called \emph{Fermat covers}. It is worth pointing out that in \cite[Kapitel 5]{BHH87} the authors provided a list of the so-called weighted line arrangements leading to ball-quotients via \emph{good abelian covers}, for an outline in English we refer to \cite{Hirzebruch85}. Our idea is to extend investigations of Hirzebruch's school to a larger class of surfaces and curve arrangements hoping that this would lead to new examples of ball-quotients. In the first part of the paper, we apply Hirzebruch's construction to rational section arrangements in Hirzebruch surfaces and we show that it is not possible to construct new examples of ball-quotients. Let us point out here that by Hirzebruch surfaces we mean classical ones and these surfaces have nothing to do with the resulting surfaces in Hirzebruch's construction. As we will be able to observe, the reason behind the claim is purely combinatorial, so our problem does not touch the question whether a certain section arrangement can be geometrically realized. In the second part, we present a general result which tells us that smooth algebraic surfaces $W$ with $K_{W}$ nef and effective and certain curve arrangements coming from ample and effective linear systems do not provide new examples of ball-quotients. At last, we provide an improved version of \cite[Theorem 4.2]{Pokora} which shows that there is in fact \emph{no} $d$-configuration such that the associated Hirzebruch-Kummer cover leads to a new example of ball-quotients.

We briefly present the main construction due to Hirzebruch focusing on the case of rational curves in Hirzebruch surfaces -- this was done in detail, for instance, in the master thesis by S. Eterovi\'{c} \cite{Eterovic}. However, since Hirzebruch's construction is the Swiss Army Knife in our toolbox, we decided to present a concise introduction to this topic. The whole theory which stands behind this paper can be found in the classical textbook \cite{BHH87}, a general construction of abelian covers for algebraic varieties can be found in \cite{Pardini}. It is worth mentioning that a similar construction to those discovered by Hirzebruch was also performed in \cite{Roulleau1}. Finally, we would like warmly encourage the reader to look at fantastic lecture notes by F. Catanese \cite{Catanese}, where the author in Section 4 is presenting a topological side of Hirzebruch-Kummer covers in the case of line arrangements (and more).

We work exclusively over the complex numbers.
\section{Hirzebruch-Kummer covers and Hirzebruch surfaces}
Let us denote by $\mathbb{F}_{e}$ the $e$-th Hirzebruch surface. The Picard group of $\mathbb{F}_{e}$ is generated by $\Gamma$ and $F$ with $\Gamma^{2} = -e$, $\Gamma.F = 1$, and $F^{2}=0$. In this and the next section, we will consider only rational section arrangements that we are going to define right now.
\begin{definition}
Let $\mathcal{C} = \{C_{1}, ..., C_{k}\}$ with $C_{i} \in |(e+1)F+\Gamma|$ and $k \geq 5$. We say that $\mathcal{C}$ is a rational section arrangement if:
\begin{itemize}
\item all curves $C_{i}$ are irreducible and smooth;
\item all intersection points are transversal;
\item all singular points of $\mathcal{C}$ have multiplicities strictly less than $k-2$. 
\end{itemize}
\end{definition}
Let us emphasize that the name of these arrangements comes from the fact that all irreducible components are smooth rational curves.
Now we introduce the following combinatorial notion. We denote by $t_{r} = t_{r}(\mathcal{C})$ the number of $r$-fold points, i.e., points where exactly $r$ curves from $\mathcal{C}$ intersect, and additionally for $i \in \{0,1,2\}$ we define
$$f_{i} = f_{i}(\mathcal{C}) = \sum_{r\geq 2}r^{i}t_{r}.$$
We start with the following combinatorial result in the spirit of \cite[Lemma 13.1]{Eterovic}.
\begin{lemma}
For a rational section arrangement $\mathcal{C}$ one has $f_{0} \geq e + 6$.
\end{lemma}
\begin{proof}
Suppose that $f_{0} < e + 6$. Then by the combinatorial equality
$$(e+2)(k^2 - k) = \sum_{r \geq 2}r(r-1)t_{r} \leq (k-3)(k-4)f_{0} < (k-3)(k-4)(e+6).$$
This leads to 
$$(e+2)(3k-6) < 2(k-3)(k-4).$$
Since $k \leq f_{0} < e+6$, where the first inequality is a consequence of \cite[Remark~7.4]{Gian}, we obtain
$$k-4 \leq e + 2 < \frac{2(k-3)(k-4)}{3k-6}.$$
Thus
$$ 0 < \frac{-k(k-4)}{3k-6}$$
and since $k\geq 5$ we arrive at a contradiction.
\end{proof}
Now we are going to provide a short outline of Hirzebruch's construction. This can be done, of course, in a general setting, but we believe that it would not lead to confusion once we explain this construction using a particular example. Given a rational section arrangement $\mathcal{C}$, denote by $s_{i}$ the section defining $C_{i}\in \mathcal{C}$, and consider the following morphisms:
$$f : \mathbb{F}_{e} \ni x \rightarrow (s_{0}(x) : ... : s_{k-1}(x)) \in \mathbb{P}^{k-1};$$
$$p: \mathbb{P}^{k-1} \ni (x_{0}: ... : x_{k-1}) \rightarrow (x_{0}^{n} : ... : x_{k-1}^{n}) \in \mathbb{P}^{k-1} \text{ with } n \geq 2.$$
Now we define the fiber product $X :=\{(x,y) \in \mathbb{F}_{e} \times \mathbb{P}^{k-1}: f(x) = p(y) \}$ which is a normal surface, and we obtain a morphism $\pi: X \rightarrow \mathbb{F}_{e}$ whose branch locus is $\mathcal{C}$. After the resolution of singularities $\rho: Y \rightarrow X$ we obtain a smooth projective surface which is called the Hirzebruch-Kummer cover of order $n^{k-1}$ and exponent $n\geq 2$. Notice that $q \in X$ is a singular point iff the multiplicity of $\pi(q)$ is greater equal to $3$. Let us denote by $\tau: Z \rightarrow \mathbb{F}_{e}$ the blowing-up along all singular points of $\mathcal{C}$ such that their multiplicities are greater equal to $3$, then there exists a morphism $\sigma: Y \rightarrow Z$ such that $\pi\rho = \tau \sigma$. Let $p = \pi(q)$ be a singular point of $\mathcal{C}$ having multiplicity $\geq 3$, and let us denote by $E_{p}$ the exceptional curve in $Z$ over $p$. Now we are ready to compute the Chern numbers. Since this procedure is well-studied (for instance \cite{BHH87,Pokora}), we present this part in a rather sketchy way. It is easy to compute the second Chern number of $Y$ (please notice that $e(C_{i}) = 2$):
$$e(Y)/n^{k-3} = n^{2}(4-2k+f_{1}-f_{0})+2n(k-f_{1}+f_{0}) + f_{1}-t_{2}.$$
Now we want to compute the first Chern number of $Y$, which is equal to $c_{1}^{2}(Y) = n^{k-1}D^{2}$
with 
$$D = \tau^{*}\bigg(K_{\mathbb{F}_{e}} + \frac{n-1}{n}C\bigg) + \sum_{p} \bigg( \frac{2n-1}{n}- \frac{n-1}{n}r_{p}\bigg)E_{p},$$
where $C = C_{1} + ... + C_{k}$, $r_{p}$ denotes the multiplicity of a singular point $p \in {\rm Sing}(\mathcal{C})$, i.e., the number of curves from $\mathcal{C}$ passing through $p$, and the above sum goes through all essential singularities of $\mathcal{C}$, i.e., those singular points which have multiplicity $\geq 3$. Quite tedious computations, with make use of the combinatorial equality $(e+2)(k^{2}-k) = f_{2}-f_{1}$, lead to 
$$c_{1}^{2}(Y)/n^{k-3}= n^{2}(8-(e+6)k+3f_{1}-4f_{0}) + 4n(k-f_{1}+f_{0}) + (e+2)k + f_{1}-f_{0}+t_{2}.$$

Now our aim is to show that for $n \geq2$, $k\geq 5$, and $e\geq 2$, the Kodaira dimension of $Y$ is non-negative, which will allow us to use the Bogomolov-Miyaoka-Yau inequality -- this part is also in the spirit of \cite{Eterovic}. However, we need to assume additionally the following property:
$$(\bullet) \text{ in } \mathcal{C} \text{ one can find exactly four sections intersecting only at double and triple points}.$$
First of all, observe that $D$ can be written as
$$D = (e+2)F + \sum_{p}a_{p}E_{p} + \sum_{j}b_{j}C_{j}^{'},$$
where $C_{j}^{'} = \tau^{*}C_{j} - \sum_{p \in {\rm EssSing}(\mathcal{C})\cap C_{j}}E_{p}$ is the strict transform of $C_{j}$, and the crucial thing is that the coefficients $a_{p}$ and $b_{j}$ are non-negative, namely
$$a_{p} \geq \frac{2n-1}{n} - \frac{3}{2}, \quad b_{j} \geq \frac{n-1}{n}-\frac{1}{2}.$$
Let us emphasize directly that in the case of $a_{p}$'s we rigorously used the fact that we can find four sections in $\mathcal{C}$ not intersecting along {\em quadruple} points. The above considerations allow us to conclude that $D$ can be written as a linear combination of effective divisors with non-negative coefficients. Additionally, it is easy to see that $D.E_{p} \geq 0$, and we need to show that if $n\geq 2$ and $e\geq 2$ one has $D.C_{j}^{'} \geq 0$.
Observe that
$$D.C_{j}^{'} = -e - 4 + \frac{n-1}{n}k(e+2) - \sum_{p \in C_{j}, r_{p}\geq 3} \bigg(\frac{n-1}{n}(r_{p}-1) - 1\bigg).$$
Using the following combinatorial count
$$(e+2)(k-1) = \sum_{p \in {\rm Sing}(\mathcal{C})\cap C_{j}}(r_{p}-1)$$ 
we obtain
$$\sum_{p \in C_{j}, r_{p}\geq 3} \bigg(\frac{n-1}{n}(r_{p}-1) - 1\bigg) = \sum_{p \in C_{j}, r_{p}\geq 3} \bigg(\frac{n-1}{n}(r_{p} - 1) \bigg)  - \#| p \in C_{j} : r_{p} \geq 3|$$
$$ = \frac{n-1}{n}(e+2)(k-1) - \#| p \in C_{j} : r_{p} \geq 3| - \frac{n-1}{n}\#| p \in C_{j} : r_{p} =2|.$$
This leads to
$$D.C_{j}^{'} \geq -2 - \frac{e+2}{n} + \frac{n-1}{n} \#| p \in C_{j} : r_{p} \geq 2|.$$
Since $\#| p \in C_{j} : r_{p} \geq 2| \geq e + 6$ we get $D.C_{j}^{'} \geq 0$. 

Now we are allowed to use the Bogomolov-Miyaoka-Yau inequality. Let us define the following Hirzebruch polynomial:
$$H_{\mathcal{C}}(n) = \frac{3e(Y) - c_{1}^{2}(Y)}{n^{k-3}} = n^{2}(4 + ek+f_{0}) + 2n(k-f_{1}+f_{0}) + 2f_{1}-4t_{2}+f_{0} -(e+2)k,$$
and by the previous considerations, if $\mathcal{C}$ is a rational section arrangement with $k\geq 5$ satisfying additionally $(\bullet)$, $n\geq 2$, and $e\geq 2$, then $H_{\mathcal{C}}(n) \geq 0$. If we can find a section arrangement $\mathcal{C}'$ with $(\bullet)$ such that $H_{\mathcal{C}'}(n_{0}) = 0$ for some $n_{0} \geq 2$, then the associated Hirzebruch-Kummer cover is a ball-quotient.

\section{Rational section arrangements in Hirzebruch surfaces}
In this section, we check whether there exists a rational section arrangement $\mathcal{C}$ such that the associated Hirzebruch-Kummer cover $Y$ is a ball-quotient. It turns out that the answer is negative. In order to observe this phenomenon, we will use the theory of constantly branched covers which was developed in \cite{BHH87}. Let us recall some facts from \cite[Section 1.3]{BHH87}. We know that if $Y$ is a ball-quotient, then all irreducible components of the (reduced) ramification divisor $\sigma^{*}(\bar{\mathcal{C}})$, where $\bar{\mathcal{C}}$ is the total transform of $\mathcal{C}$ in $Z$, must satisfy $\text{prop}(E) = 2E^{2} - e(E) = 0$. In particular, each irreducible component $C$ of $\sigma^{*}E_{p}$ satisfies $C^{2} = -n^{r_{p}-2}$ and
\begin{equation}
\label{firstprop}
\text{prop}(C) = n^{r_{p}-2}((r_{p}-2)(n-1)-4).
\end{equation}
The condition $\text{prop}(C) = 0$ leads to
$$(n-1)(r_{p}-2) = 4,$$ and the following pairs are admissible (we have the following order of listing: $(n,r_{p})$):
$$(5,3), \quad (3,4), \quad (2,6).$$
It means that $Y$ can be a ball-quotient if one of the following conditions is satisfied:
\begin{itemize}
\item $t_{r} = 0$ for $r \neq 2, 3$ and $n = 5$,
\item $t_{r} = 0$ for $r \neq 2, 4$ and $n = 3$,
\item $t_{r} = 0$ for $r \neq 2, 6$ and $n = 2$,
\end{itemize}
The above considerations suffice to conclude that we cannot construct new examples of ball-quotients for $n \in \{3,5\}$. Let us compute $H_{\mathcal{C}}(3)$ and $H_{\mathcal{C}}(5)$. After a moment we see that:
$$H_{\mathcal{C}}(3) = 36 + 8ek + 4k - 4f_{1} + 16f_{0}-4t_{2} \geq 0,$$
$$H_{\mathcal{C}}(5) = 100 +24ek +8k + 36f_{0} - 8f_{1} - 4t_{2} \geq 0.$$
Consider the case $n=3$ and $\mathcal{C}$ having only double and quadruple points. If the associated Hirzebruch-Kummer cover $Y$ is a ball-quotient, then in particular
$$ 0 = H_{\mathcal{C}}(3) = 9 + (2e+1)k + t_{2} > 0,$$
a contradiction.
Similarly, if $n=5$ and $\mathcal{C}$ has only double and triple points, then $Y$ is a ball-quotient if
$$0 = H_{\mathcal{C}}(5) = 25 + (6e+2)k+4t_{2} + 3t_{3} >0,$$
a contradiction.

Now we need to consider the case $n=2$. We have
$$H_{\mathcal{C}}(2) = 16 + 3ek + 2k + 9f_{0}-2f_{1} - 4t_{2} \geq 0.$$
If there exists a rational section arrangement $\mathcal{C}'$ having only double and sixfold points such that the associated Hirzebruch-Kummer cover is a ball-quotient, then
$$16 + (3e+2)k + t_{2} = 3t_{6},$$
We need to find some constraints on the number of double and sixfold points of $\mathcal{C}'$. Here we can use in fact a general statement from Section 4 (positivity of the canonical divisor in this place does not play any role). In order to avoid repetitions, let us briefly conclude that using formulae (\ref{constraints}) with $a = e+2$ and $b=-e-4$ we finally obtain
$$-8 = (e+1)k,$$
a contradiction.
\begin{theorem}
There does not exist any section arrangement $\mathcal{C} \subset \mathbb{F}_{e}$ with $e \geq 2$ such that the associated Hirzebruch-Kummer cover is a ball-quotient.
\end{theorem}
Our choice of curves that we applied in Hirzebruch-Kummer's construction in Section 2 and Section 3 is not accidental. In the next section, we are going to give some evidence that we should not expect to obtain new examples of ball-quotients using curves of positive genera.
\section{Curve arrangements in surfaces with nef and effective canonical divisor}
In this section, we consider only pairs $(W,\mathcal{C})$, where $W$ is an algebraic surface and $\mathcal{C}$ an arrangement of curves, admitting Hirzebruch-Kummer covers. We assume that our surface $W$ is smooth, complex, and projective such that $K_{W}$ is nef and effective. We will keep the same notation as in Section 2.
\begin{definition}
Assume that the linear system $|A|$ is effective for a certain ample line bundle $A$ on $Z$ and denote by $C_{1}, ..., C_{k} \in |A|$ smooth irreducible curves. We say that $\mathcal{C}= \{C_{1}, ...,C_{k}\}$ is a regular arrangement of curves if
\begin{itemize}
\item all intersection points are transversal;
\item there is no point where all curves from $\mathcal{C}$ meet.
\end{itemize}
\end{definition}

Let us recall that for such configurations the following combinatorial identities hold:
\begin{equation}
\label{first}
a(k^{2}-k) = \sum_{r\geq 2}(r^{2}-r)t_{r},
\end{equation}
\begin{equation}
\label{second}
a(k-1) = \sum_{p \in {\rm Sing}(\mathcal{C})\cap C_{j}} (r_{p}-1),
\end{equation}
where $a = A^{2}$ and $C_{j} \in \mathcal{C}$ is fixed. Moreover, we define $b := K_{W}.C_{j}$, and obviously $b \geq 0$.
Now we consider the abelian cover of $W$ branched along a regular arrangement $\mathcal{C}$, which can be viewed, according to the previous sections, as the minimal desingularization $Y$ of the fiber product $W \times_{\mathbb{P}^{k-1}}\mathbb{P}^{k-1}$. We can compute the Chern numbers of $Y$, namely
$$c_{2}(Y)/n^{k-3} = n^{2}(e(W) + (a+b)k+f_{1}-f_{0})+n(-(a+b)-2f_{1}+2f_{0})+f_{1}-t_{2};$$
$$c^{2}_{1}(Y)/n^{k-3} = n^{2}(K_{W}^{2} + (a+2b)k+3f_{1}-4f_{0}) +2n(-(a+b)k-2f_{1}+2f_{0})+ak+f_{1}-f_{0}+t_{2}.$$

Let us observe that $c_{1}^{2}(Y) = n^{k -1}K^{2}$ with
$$K = \tau^{*}(K_{W}) + \sum E_{P} + \frac{n-1}{n}\bigg( \sum E_{P} + \tau^{*}C - \sum r_{P}E_{P}\bigg),$$
where $C = C_{1} + ... + C_{k}$, which shows that $K$ is effective.
It means that we are allowed to use the Bogomolov-Miyaoka-Yau inequality for any $n\geq 2$. Let us define the following Hirzebruch polynomial:
\begin{multline}
H_{\mathcal{C}}(n) = \frac{3c_{2}(Y) - c_{1}^{2}(Y)}{n^{k - 3}} = n^{2}(3e(W)-K^{2}_{W} + (2a+b)k + f_{0}) + n(-(a+b)k -2f_{1} + 2f_{0}) \\-ak +2f_{1}+f_{0}-4t_{2}.
\end{multline}

Now we would like to check whether there exists a regular arrangement $\mathcal{C}$ such that the associated Hirzebruch-Kummer cover provides a new example of ball-quotients. We need to consider cases for $n \in \{2,3,5\}$. We can start with $n=3$. For brevity, let us denote by $\delta(W) : = 3e(W) - K^{2}_{W}$. We see that $H_{\mathcal{C}}(3) = 0$ leads to 
$$\frac{9}{4}\delta(W) + \bigg(\frac{7}{2}a + \frac{3}{2}b\bigg)k + t_{2} + t_{3} = \sum_{r \geq 5}(r-4)t_{r},$$
and if there exists a regular arrangement with double and quadruple points providing a ball-quotient, then
$$0 = \frac{9}{4}\delta(W) + \bigg(\frac{7}{2}a + \frac{3}{2}b\bigg)k  + t_{2}  > 0,$$
a contradiction.

Similarly, if $n=5$, then $H_{\mathcal{C}}(5) = 0$ leads to 
$$\frac{25}{4}\delta(W) + (11a+5b)k+4t_{2} + 3t_{3} + t_{4} = \sum_{r\geq 5}(2r-9)t_{r},$$
and if there exists a regular configuration with only double and triple points providing a ball-quotient, then
$$0 = \frac{25}{4}\delta(W) + (11a+5b)k +4t_{2} + 3t_{3} >0,$$
a contradiction.

Now we need to deal with the case $n=2$. The condition $H_{\mathcal{C}}(2) = 0$ leads to
$$4\delta(W) + (5a+2b)k + t_{2} + 3t_{3} + t_{4} = \sum_{r\geq 5}(2r-9)t_{r},$$
and if there exists a regular configuration $\mathcal{C}$ having double and sixfold points providing a ball-quotient, then
\begin{equation}
\label{secondprop}
4\delta(W) + (5a+2b)k  + t_{2} = 3t_{6}.
\end{equation}
We need to find combinatorial constraints on the number of double and sixfold points for such regular arrangements. This can be done using an extended version of \emph{deviations from proportionality} -- see for instance \cite{Pokora} for details. It means that we need to investigate the condition ${\rm prop}(D_{j}) = 0$ with $D_{j} = \sigma^{*}(C_{j}')$, where $C_{j}'$ is the strict transform of $C_{j} \in \mathcal{C}$ under the blowing-up $\tau$. This gives
$$ 0 = {\rm prop}(D_{j}) = n^{k-3}({\rm prop}(C_{j}) + (n-1)(r_{j}-e(C_{j}))-2\gamma_{j}),$$
where $r_{j}$ denotes the number of singular points in $C_{j}$, $\gamma_{j}$ denotes the number of essential singular points, i.e., those with multiplicities $\geq 3$, and $r_{j,2}$ denotes the number of double points (altogether $r_{j} = \gamma_{j} + r_{j,2}$).

For $n=2$, we have only double and sixfold points, and we need to use the conditions ${\rm prop}(C) = 0$ and ${\rm prop}(D_{j}) = 0$ simultaneously. These two conditions lead us to 
$$r_{j,6} + r_{j,2} = \frac{a(k-1) -8a - 4b}{3}.$$
It is worth pointing out that during computations we have used (\ref{second}). Since
$$\sum_{j} r_{j,r}= rt_{r},$$
we obtain the following formula:
$$2t_{2} + 6t_{6} = f_{1} = k \cdot \frac{a(k-1) -8a - 4b}{3}.$$
Combining this with the combinatorial equality (\ref{first}), we can find constraints on $t_{2}$ and $t_{6}$, namely:
\begin{equation}
\label{constraints}
t_{2} = \frac{ak^{2}-21ak-10bk}{12} \text{ , } t_{6} = \frac{ak^{2}+3ak+2bk}{36}.
\end{equation}
Plugging these values to $(\ref{secondprop})$, we obtain
$$0< (3a+b)k = -4\delta(W) \leq 0$$
a contradiction.

We have shown the following theorem.
\begin{theorem}
Let $W$ be a smooth complex projective surface with $K_{W}$ nef and effective, and $\mathcal{C}$ a regular arrangement of $k\geq 5$ curves. Then the associated Hirzebruch-Kummer cover is never a ball-quotient.
\end{theorem}
It is natural to wonder whether the assumption for curves being members of effective and ample linear systems is optimal. Let us recall the following example which can be found in \cite{H84}.
\begin{example}
We consider $A = T \times T$ an abelian surface which is the product of two elliptic curves with complex multiplication. We consider the arrangement $\mathcal{L}_{1} = \{F_{1},F_{2},\triangle , G\}$, where $F_{1}, F_{2}$ are two fibers, $\triangle$ is the diagonal, and $G$ is the graph of complex multiplication. It is easy to observe that these four curves intersect exactly at one quadruple point, and the Hirzebruch-Kummer cover of $A$ branched along $\mathcal{L}_{1}$ provides a surface $Y$ for which we have $c_{1}^{2}(Y) = 3e(Y)$, so we obtained a ball-quotient.
\end{example}
\section{Addendum to \cite[Theorem 4.2]{Pokora}}
In this addendum, we would like to improve one of the author's results from \cite{Pokora}. Before we present the improvement, we recall one definition from \cite{Pokora} in order to keep the coherence with the mentioned article.
\begin{definition}
An arrangement $\mathcal{C} = \{C_{1}, ..., C_{k}\} \subset \mathbb{P}^{2}$ is called $d$-configuration if
\begin{itemize}
\item all curves $C_{i}$ are smooth of degree $d \geq 2$;
\item all intersection points are transversal;
\item $t_{k}=0$.
\end{itemize}
\end{definition}
Now our aim is to show the following theorem (cf. \cite[Theoerm 4.2]{Pokora}).
\begin{theorem}
There does not exist any $d$-configuration of curves in $\mathbb{P}^{2}$ with $d\geq 3$ such that the associated Hirzebruch-Kummer cover is a ball-quotient.
\end{theorem}
\begin{proof}
In order to prove the theorem, we need to exclude the last remaining combinatorial case from \cite{Pokora}, namely $n=2$ and $d$-configurations with $d\geq 3$ having the following combinatorics:
$$t_{2} = \frac{dk(dk-21d+30)}{12}, \quad t_{6} = \frac{dk(dk+3d-6)}{36}.$$
Suppose that $\mathcal{C}$ as above leads to a ball-quotient $Y$. Then
$$0 = 3e(Y)-c_{1}^{2}(Y) = (5d^{2}-6d)k + t_{2} - 3t_{6}.$$
This gives
$$(5d^{2}-6d)k + \frac{dk(dk-21d+30)}{12} = 3 \cdot \frac{dk(dk+3d-6)}{36},$$
which leads to 
$$36d(d-1) = 0,$$
a contradiction.
\end{proof}
\begin{remark}
Very recently Professor Fabrizio Catanese informed me that Mara Neusel in her Diplomarbeit (unpublished) was working on the problem of the existence of ball-quotients constructed via Hirzebruch-Kummer covers of the complex projective plane branched along arrangements of conics.
\end{remark}
\section*{Acknowledgement}
I am really grateful to Giancarlo Urz\'{u}a for sharing with me the master thesis of Sebastian Eterovi\'{c}, his former student. The mentioned thesis was an inspiration for the first part of this paper. I would also like to thank Piotr Achinger, Igor Dolgachev, Micha\l \, Kapustka, Xavier Roulleau, and Giancarlo Urz\'{u}a for useful comments. Finally, I would like to thank Professor Fabrizio Catanese for enlightening conversations on Hirzebruch-Kummer covers and algebraic surfaces, and for important remarks for which I am extremely grateful. During the project the author was supported by the programme for young researchers at Institute of Mathematics Polish Academy of Sciences.

\bigskip
   Piotr Pokora,
    Institute of Mathematics,
    Polish Academy of Sciences,
    ul. \'{S}niadeckich 8,
    PL-00-656 Warszawa, Poland. \\
\nopagebreak
   \textit{E-mail address:} \texttt{piotrpkr@gmail.com}

\end{document}